\documentclass[10pt]{article}

\usepackage{amsxtra,amssymb,amsthm,amsmath,latexsym}
\usepackage{graphicx}
\usepackage{epsfig}
\usepackage{afterpage}
\newtheorem{thm}{Theorem}[section]

\newtheorem{lem}[thm]{Lemma}
\newtheorem{rem}[thm]{Remark}

 \newcommand{\thmref}[1]{Theorem~\ref{#1}}
 \newcommand{\lemref}[1]{Lemma~\ref{#1}}

\newcommand{\R}{{\mathbb R}}

\newcommand{\vp}{{\varphi}}
\newcommand{\la}{{\langle}}
\newcommand{\ra}{{\rangle}}
\newcommand{\dl}{{\delta}}

\newcommand{\bee}{\begin{equation*}}
\newcommand{\eee}{\end{equation*}}
\newcommand{\be}{\begin{equation}}
\newcommand{\ee}{\end{equation}}
\newcommand{\pn}{\par\noindent}
\title{A collocation method for solving some integral equations in
distributions}
\author{Sapto W. Indratno\\
\small Department of Mathematics\\[-0.8ex]
\small Kansas State University, Manhattan, KS 66506-2602, USA\\
\small Department of Mathematics\\[-0.8ex]
\small Bandung Institute of Technology, Bandung, Indonesia\\
\small \texttt{sapto@math.itb.ac.id}
\and
A G Ramm\\
\small Department of Mathematics\\[-0.8ex]
\small Kansas State University, Manhattan, KS 66506-2602, USA\\[-0.8ex]
\small \texttt{ramm@math.ksu.edu}\\
}

\begin{document}
\date{}
\maketitle

\begin{abstract}
A collocation method is presented for numerical solution of a
typical integral equation $Rh:=\int_D R(x,y)h(y)dy=f(x),\quad x\in
\overline{D}$ of the class $\mathcal{R}$, whose kernels are of
positive rational functions of arbitrary selfadjoint elliptic
operators defined in the whole space $\R^r$, and $D\subset \R^r$ is
a bounded domain. Several numerical examples are given to
demonstrate the efficiency and stability of the proposed method.
\end{abstract}
\pn{\\ {\em MSC: 45A05, 45P05, 46F05, 62M40, 65R20, 74H15 }   \\
{\em Key words:} integral equations in distributions, signal
estimation, collocation method.}

\section{Introduction}
In \cite{RAMM1} a general theory of integral equations of the class
$\mathcal{R}$ was developed. The integral equations of the class
$\mathcal{R}$ are written in the following form: \be\label{R1}
Rh:=\int_D R(x,y)h(y)dy=f(x),\quad x\in \overline{D}:=D\cup \Gamma,
\ee where $D\in \R^r$ is a (bounded) domain with a (smooth) boundary
$\Gamma$. Here the kernel $R(x,y)$ has the following form
\cite{RAMM1,R486,RAMM3,RAMM4}: \be\label{R2}
R(x,y)=\int_{\Lambda}P(\lambda)Q^{-1}(\lambda)\Phi(x,y,\lambda)d\rho(\lambda),
\ee where $P(\lambda)$, $Q(\lambda)>0$ are polynomials, deg$P=p$,
deg$Q=q$, $q>p$, and $\Phi$, $\rho$, $\Lambda$ are the spectral
kernel, spectral measure, and spectrum of a selfadjoint elliptic
operator $\mathcal{L}$ on $L^2(\R^r)$ of order $s$. It was also
proved in \cite{RAMM1} that $R: \dot{H}^{-\alpha}(D)\to H^\alpha(D)$
is an isomorphism, where $H^\alpha(D)$ is the Sobolev space and
$\dot{H}^{-\alpha}(D)$ its dual space with respect to the $L^2(D)$
inner product, $\alpha=\frac {s(q-p)}{2}$. Here the space $\dot{H}^{-\alpha}(D)$ consists of distributions in $H^\alpha(\R^r)$ with support in the closure of $D$. In this paper we consider
a particular type of integral equations of the class $\mathcal{R}$
with $D=(-1,1)$, $r=1$, $\mathcal{L}=-i\partial$,
$\partial:=\frac{d}{dx}$, $\Lambda\in(-\infty,\infty)$,
$d\rho(\lambda)=d\lambda$,
$\Phi(x,y,\lambda)=\frac{e^{i\lambda(x-y)}}{2\pi}$, $P(\lambda)=1$,
$Q(\lambda)=\frac{\lambda^2+1}{2}$, $s=1$, $p=0$, $q=2$ and
$\alpha=1$, i.e., \be\label{e1}
Rh(x):=\int_{-1}^1e^{-|x-y|}h(y)dy=f(x), \ee where $h\in
\dot{H}^{-1}[-1,1]$ and $f\in H^{1}[-1,1]$. We denote the inner
product and norm in $H^1[-1,1]$ by \be\label{e2} \la
u,v\ra_1:=\int_{-1}^1\left(u(x)\overline{v(x)}+
u'(x)\overline{v'(x)}dx\right) dx \quad u,v\in H^1([-1,1]),\ee and
\be\label{e3} \|u\|_1^2:=\int_{-1}^1\left(|u(x)|^2+
|u'(x)|^2\right)dx, \ee respectively, where the primes denote
derivatives and the bar stands for complex conjugate. If $u$ and $v$
are real-valued functions in $H^1[-1,1]$ then the bar notations
given in \eqref{e2} can be dropped. Note that if $f$ is a complex
valued function then solving equation \eqref{e1} is equivalent to
solving the equations: \be
\int_{-1}^1e^{-|x-y|}h_k(y)dy=f_k(x),\quad k=1,2,\ee where
$h_1(x):=$Re$h(x)$, $h_2(x):=$Im$h(x)$, $f_1(x):=$Re$f(x)$,
$f_2(x):=$Im$f(x)$ and $h(x)=h_1(x)+ih_2(x)$, $i=\sqrt{-1}$.
Therefore, without loss of generality we assume throughout that
$f(x)$ is real-valued.

It was proved in \cite{ R486} that the operator $R$ defined in
\eqref{e1} is an isomorphism between $\dot{H}^{-1}[-1,1]$ and
$H^1[-1,1]$. Therefore, problem \eqref{e1} is well posed in the
sense that small changes in the data $f(x)$ in the $H^1[-1,1]$
norm will result in small in $\dot{H}^{-1}[-1,1]$ norm changes to
the solution $h(y)$. Moreover, the solution to \eqref{e1} can be
written in the following form: \be\label{e4}
h(x)=a_{-1}\dl(x+1)+a_0\dl(x-1)+g(x),\ee where \be\label{e5}
a_{-1}:=\frac{f(-1)-f'(-1)}{2},\quad a_0:=\frac{f'(1)+f(1)}{2},\ee
\be\label{e6} g(x):=\frac{-f''(x)+f(x)}{2}, \ee and $\dl(x)$ is the
delta function. Here and throughout this paper we assume that $f\in
C^\alpha[-1,1]$, $\alpha\geq 2$. It follows from \eqref{e5} that
$h(x)=g(x)$ if and only if $f(-1)=f'(-1)$ and $f(1)=-f'(1)$.

In \cite{RAMM3,RAMM4} the problem of solving equation \eqref{e1}
numerically have been posed and solved. The least squares method was
used in these papers. The goal of this paper is to develop a version
of the collocation method which can be applied easily and
numerically efficiently. In \cite{RAMM5} some basic ideas
for using collocation method are proposed. In this paper some of these
ideas are used and new ideas, related to the choice of the basis
functions, are introduced. In this paper the emphasis is on the
development of methodology for solving basic equation (1) of the
estimation theory by a version of the collocation method.
The novelty of this version consists in minimization of a
discrepancy functional (26), see below. This methodology is illustrated
by a detailed analysis applied to solving equation (3), but it is
applicable to general equations of the class $\mathcal{R}$. One of the
goals of this paper is to demonstrate that collocation method can
be successfully applied to numerical solution of some integral equations
whose solutions are distributions, provided that the theoretical analysis
gives sufficient information about the singular part of the solutions.

Since $f\in C^\alpha[-1,1]$, $\alpha\geq
2$, it follows from \eqref{e6} that $g\in C[-1,1]$. Therefore,
there exist basis functions $\vp_j(x)\in C[-1,1]$,
$j=1,2,\hdots,m$, such that \be\label{e6a}
\max_{x\in[-1,1]}|g(x)-g_m(x)|\to 0\quad \text{as }m\to \infty, \ee
where \be\label{e6b} g_m(x):=\sum_{j=1}^mc_j^{(m)}\vp_j(x), \ee
$c^{(m)}_j$, $j=1,2,\hdots,m$, are constants. Hence the approximate
solution of equation \eqref{e1} can be represented by \be\label{e7}
h_m(x)=c_{-1}^{(m)}\dl(x+1)+c_0^{(m)}\dl(x-1)+g_m(x), \ee where
$c_j^{(m)}$, $j=-1,0$, are constants and $g_m(x)$ is defined in
\eqref{e6b}. The basis functions $\vp_j$ play an important role in
our method. It is proved in Section 3 that the basis functions
$\vp_j$ in \eqref{e7} can be chosen from the linear B-splines. The
usage of the linear B-splines reduces the computation time, because
computing \eqref{e7} at a particular point $x$ requires at most two
out of the $m$ basis functions $\vp_j$. For a more detailed
discussion of the family of B-splines we refer to \cite{LLS}. In
Section 2 we derive a method for obtaining the coefficients
$c_j^{(m)}$, $j=-1,0,1,\hdots,m$, given in \eqref{e7}. This method
is based on solving a finite-dimensional least squares problem ( see
equation \eqref{e26} below ) and differs from the usual collocation
method discussed in \cite{KA} and \cite{MP}. We approximate
$\|f-Rh_m\|_1^2$ by a quadrature formula.
The resulting finite-dimensional linear
algebraic system depends on the choice of the basis functions. Using
linear B-splines as the basis functions, we prove the existence and
uniqueness of the solution to this linear algebraic system for all
$m=m(n)$ depending on the number $n$ of collocation points used in the
left rectangle quadrature rule. The convergence of our collocation method
is proved
in this Section. An example of the choice of the basis functions which
yields the convergence of our version of the collocation method is given
in Section 3. In
Section 4  we give numerical results of applying our method to
several problems that discussed in \cite{RAMM4}.

\section{The collocation method}
In this Section we derive a collocation method for solving
equation \eqref{e1}. From equation \eqref{e1} we
get\be\label{e8}\begin{split}
Rh(x)&=a_{-1}e^{-(x+1)}+a_0e^{-(1-x)}\\
&+\left(e^{-x}\int_{-1}^xe^{y}g(y)dy+e^x\int_x^1e^{-y}g(y)dy
\right)=f(x). \end{split}\ee
Assuming that $f\in
C^2([-1,1])$ and differentiating the above equation, one obtains
\be\label{e9}\begin{split}
(Rh)'(x)&=-a_{-1}e^{-(x+1)}+a_0e^{-(1-x)}\\
&+\left(e^x\int_x^1e^{-y}g(y)dy-e^{-x}\int_{-1}^xe^{y}g(y)dy\right)=f'(x).
\end{split}\ee Thus, $f(x)$ and $f'(x)$ are continuous in the
interval $[-1,1]$.
Let us use the approximate solution given in
\eqref{e7}. From \eqref{e8}, \eqref{e9} and \eqref{e7} we obtain
\be\label{e10}\begin{split}
Rh_m(x)&=c_{-1}^{(m)}e^{-(x+1)}+c_0^{(m)}e^{-(1-x)}\\
&+\sum_{j=1}^{m}c_j^{(m)}\left(e^{-x}\int_{-1}^xe^{y}\vp_j(y)dy+e^x\int_x^1e^{-y}\vp_j(y)dy
\right):=f_m(x),
\end{split}\ee and
\be\label{e11}\begin{split}
(Rh_m)'(x)&=-c_{-1}^{(m)}e^{-(x+1)}+c_0^{(m)}e^{-(1-x)}\\
&+\sum_{j=1}^{m}c_j^{(m)}\left[e^x\int_x^1e^{-y}\vp_j(y)dy-e^{-x}\int_{-1}^xe^{y}\vp_j(y)dy\right]:=(f_m)'(x).
\end{split}\ee Thus,  $Rh_m(x)$ and $(Rh_m)'(x)$ are continuous in the
interval $[-1,1]$. Since $f(x)-Rh_m(x)$ and $f'(x)-(Rh_m)'(x)$ are
continuous in the interval $[-1,1]$, we may assume throughout that
the functions \be\label{e12} J_{1,m}:=[f(x)-Rh_m(x)]^2\ee and
\be\label{e13} J_{2,m}:=[f'(x)-(Rh_m)'(x)]^2\ee are
Riemann-integrable over the interval $[-1,1]$.

Let us define
\be\begin{split}\label{e14} q_{-1}(x)&:=e^{-(x+1)},\ q_0(x):=e^{-(1-x)},\\
q_j(x)&:=\int_{-1}^1e^{-|x-y|}\vp_j(y)dy,\quad j=1,2,\hdots,m,
\end{split}\ee and define a mapping $\mathcal{H}_n:C^2[-1,1]\to
\R^{n}_1$ by the formula:
\be\label{e15} \mathcal{H}_n\phi=\left(
                                            \begin{array}{c}
                                             \phi(x_1) \\
                                             \phi(x_2) \\
                                              \vdots \\
                                              \phi(x_n)\\
                                            \end{array}
                                          \right),\ \phi(x) \in
C^2[-1,1],\ee where \be\label{e16}\R^{n}_1:=\left\{\left(
                                            \begin{array}{c}
                                              z_1 \\
                                              z_2 \\
                                              \vdots \\
                                              z_{n} \\
                                            \end{array}
                                          \right)\in \R^n \ |\ z_j:=z(x_j),\ z\in
C^2[-1,1]\right\}\ee and $x_j$ are some collocation points which
will be chosen later. We equip the space $\R^{n}_1$ with the
following inner product and norm
 \be\label{e17} \la
u,v\ra_{w^{(n)},1}:=\sum_{j=1}^nw_j^{(n)}(u_jv_j+u'_jv'_j),\quad
u,v\in \R^n_1,\ee \be\label{e18}
\|u\|_{w^{(n)},1}^2:=\sum_{j=1}^nw_j^{(n)}[u_j^2+(u'_j)^2],\quad
u\in \R^n_1, \ee respectively, where $u_j:=u(x_j)$, $u'_j:=u'(x_j)$,
$v_j:=v(x_j)$, $v'_j:=v'(x_j)$, and $w_j>0$ are some quadrature
weights corresponding to the collocation points $x_j$, $j=1,2,\hdots,n$.

Applying $\mathcal{H}_n$ to $Rh_m$, one
gets\be\label{e19}\begin{split}
(\mathcal{H}_nRh_m)_i&=c_{-1}^{(m)}e^{-(1+x_i)}+c_0^{(m)}e^{-(1-x_i)}+
\sum_{j=1}^{m}c_j^{(m)}\int_{-1}^{x_i}e^{-(x_i-y)}\vp_j(y)dy+\\
&\sum_{j=1}^{m}c_j^{(m)}\int_{x_i}^{1}e^{-(y-x_i)}\vp_j(y)dy,\
i=1,2,\hdots,n, \end{split}\ee where $m=m(n)$ is an integer depending on
$n$ such that \be\label{e20} m(n)+2\leq n, \ \lim_{n\to\infty}m(n)=\infty.
\ee

Let \be\label{e21}
G_n(c^{(m)}):=\|\mathcal{H}_{n}(f-Rh_m)\|_{w^{(n)},1}^2,\ee where
$f$ and $Rh_m$ are defined in \eqref{e1} and \eqref{e10},
respectively, $\mathcal{H}_n$ defined in \eqref{e15},
$\|\cdot\|_{w^{(n)},1}$ defined in \eqref{e18} and $c^{(m)}=\left(
                                                          \begin{array}{c}
                                                            c_{-1} \\
                                                            c_0 \\
                                                            c_1 \\
                                                            \vdots \\
                                                            c_{m} \\
                                                          \end{array}
                                                        \right)
\in \R^{m+2}$. Let us choose \be\label{e22} w_j^{(n)}=\frac{2}{n},\
j=1,2,\hdots,n,\ee and \be\label{e23} x_j=-1+(j-1)s,\
s:=\frac{2}{n},\ j=1,2,\hdots,n, \ee so that
$\|\mathcal{H}_{n}(f-Rh_m)\|_{w^{(n)},1}^2$ is the left Riemannian
sum of $\|f-Rh_m\|^2_1$, i.e., \be\label{e24}
|\|f-Rh_m\|^2_1-\|\mathcal{H}_{n}(f-Rh_m)\|_{w^{(n)},1}^2|:=\dl_n\to
0\text{ as }n\to \infty. \ee

\begin{rem}
If $J_{1,m}$ and $J_{2,m}$ are in $C^2[-1,1]$, where $J_{1,m}$ and
$J_{2,m}$ are defined in \eqref{e12} and \eqref{e13}, respectively,
then one may replace the weights $w_j^{(n)}$ with the weights of the
compound trapezoidal rule, and get the estimate
\be\begin{split}\label{e25}
\dl_n&=\left|\int_{-1}^1(J_{1,m}(x)+J_{2,m}(x))dx-\sum_{j=1}^nw_j^{(n)}(J_{1,m}(x_j)+J_{2,m}(x_j))\right|\\
&\leq \frac{1}{3n^2}D_J,\end{split} \ee where $\dl_n$ is defined in
\eqref{e24} and \be
D_J:=|J'_{1,m}(1)+J'_{2,m}(1)-(J'_{1,m}(-1)+J'_{2,m}(-1) )| .\ee
Here we have used the following estimate of the compound trapezoidal
rule \cite{PDPR84, ER80}: \be
\left|\int_a^b\eta(x)dx-\sum_{j=1}^nw_j^{(n)}\eta(x_j)\right|\leq
\frac{(b-a)^2}{12n^2}\left|\int_a^b\eta''(x)dx\right|=
\frac{(b-a)^2}{12n^2}|\eta'(b)-\eta'(a)|,\ee where $\eta\in
C^2[a,b]$. Therefore, if $D_J\leq C$ for all $m$, where $C>0$ is a
constant, then $\dl_n=O\left(\frac{1}{n^2}\right).$
\end{rem}

The constants $c_j^{(m)}$ in the approximate solution $h_m$, see
\eqref{e7}, are obtained by solving the following least squares
problem: \be\label{e26} \min_{c^{(m)}}G_n(c^{(m)}),\ee where $G_n$
is defined in \eqref{e21}.

A necessary condition for the minimum in \eqref{e26} is
\be\label{e27} 0=\sum_{l=1}^{n}w_l^{(n)}\left(E_{m,l} \frac{\partial
E_{m,l}}{\partial c_k}+E_{m,l}' \frac{\partial E_{m,l}'}{\partial
c_k}\right)\quad k=-1,0,1,\hdots,m, \ee where \be\label{e28}
E_{m,l}:=(f-Rh_m)(x_l),\quad E'_{m,l}:=(f-Rh_m)'(x_l),\
l=1,2,\hdots,n. \ee

Necessary condition \eqref{e27} yields the following linear algebraic
system (LAS): \be\label{e29} A_{m+2}c^{(m)}=F_{m+2}, \ee where
$c^{(m)}\in \R^{m+2}$, $A_{m+2}$ is a square, symmetric matrix with
the following entries: \be\begin{split}\label{e30}
(A_{m+2})_{1,1}&:=2\sum_{l=1}^{n}w_l^{(n)}e^{-2(x_l+1)},\ (A_{m+2})_{1,2}=0,\\
(A_{m+2})_{1,j}&:=2\sum_{l=1}^{n}w_l^{(n)}C_{l,j-2}e^{-(x_l+1)},\
j=3,\hdots,m+2,
\end{split}\ee
\be\label{e31}\begin{split} (A_{m+2})_{2,2}&:=2\sum_{l=1}^{n}w_l^{(n)}e^{-2(1-x_l)},\\
(A_{m+2})_{2,j}&:=2\sum_{l=1}^{n}w_l^{(n)}B_{l,j-2}e^{-(1-x_l)},\
j=3,\hdots,m+2, \end{split}\ee

\be\begin{split}\label{e32}
(A_{m+2})_{i,j}&:=\sum_{l=1}^{n}2w_l^{(n)}(B_{l,i-2}B_{l,j-2}+C_{l,i-2}C_{l,j-2}),\\
i&=3,\hdots,m+2,\ j=i,\hdots,m+2,\\
(A_{m+2})_{j,i}&=(A_{m+2})_{i,j},\quad i,j=1,2,\hdots,m+2,
\end{split}\ee
$F_{m+2}$ is a vector in $\R^{m+2}$ with the following elements:
\be\begin{split}\label{e33}
(F_{m+2})_1&:=\sum_{l=1}^{n}w_l^{(n)}(f(x_l)-f'(x_l))e^{-(x_l+1)}=\la \mathcal{H}_nRh,\mathcal{H}q_{-1}\ra_{w^{(n)},1}\\
(F_{m+2})_2&:=\sum_{l=1}^{n}w_l^{(n)}(f(x_l)+f'(x_l))e^{-(1-x_l)}=\la \mathcal{H}_nRh,\mathcal{H}q_{0}\ra_{w^{(n)},1},\\
(F_{m+2})_i&:=\sum_{l=1}^{n}w_l^{(n)}[f(x_l)(C_{l,i-2}+B_{l,i-2})+f'(x_l)(B_{l,i-2}-C_{l,i-2})]\\
&=\la \mathcal{H}_nRh,\mathcal{H}q_{i}\ra_{w^{(n)},1},\
i=3,\hdots,m+2,
\end{split}\ee
and
\be\label{e34} B_{l,j}:=\int_{x_l}^1e^{-(y-x_l)}\vp_j(y)dy,\
C_{l,j}:=\int_{-1}^{x_l}e^{-(x_l-y)}\vp_j(y)dy.\ee

\begin{thm}\label{thm0}
Assume that the vectors $\mathcal{H}_{n}q_j$, $j=-1,0,1,\hdots,m$
are linearly independent. Then linear algebraic system \eqref{e29}
is uniquely solvable for all $m$, where $m$ is an integer depending
on $n$ such that \eqref{e20} holds.
\end{thm}
\begin{proof}
Consider $q_j\in H^1[-1,1]$ defined in \eqref{e14}. Using the
inner product in $\R^{n}_1$, one gets \be\label{e35}
(A_{m+2})_{i,j}=\la
\mathcal{H}_{n}q_{i-2},\mathcal{H}_nq_{j-2}\ra_{w^{(n)},1},\ \
i,j=1,2,\hdots,m+2, \ee i.e., $A_{m+2}$ is a Gram matrix. We have assumed
that the vectors $\mathcal{H}_nq_j\in \R^{n}_1$,
$j=-1,0,1,\hdots,m$, are linearly independent. Therefore, the determinant
of the matrix $A_{m+2}$ is nonzero.
This implies linear algebraic system \eqref{e29} has a unique solution.\\
\thmref{thm0} is proved.
\end{proof}

It is possible to choose basis functions $\vp_j$ such that the
vectors $\mathcal{H}_{n}q_j$, $j=-1,0,1,\hdots,m,$ are linearly
independent. An example of such choice of the basis functions is given in
Section 3.

\begin{lem}\label{lem1}
Let $y_m:=c^{(m)}_{min}$ be the unique minimizer for problem
\eqref{e26}. Then \be G_n(y_m)\to 0 \text{ as }n\to \infty, \ee
where $G_n$ is defined in \eqref{e21} and $m$ is an integer
depending on $n$ such that \eqref{e20} holds.
\end{lem}
\begin{proof}
Let $$h(x)=a_{-1}\dl(x+1)+a_0\dl(x-1)+g(x)$$ be the exact solution to
\eqref{e1}, $Rh=f$, where $g(x)\in C[-1,1]$, and define \be\label{e37}
\tilde{h}_m(x)=a_{-1}\dl(x+1)+a_0\dl(x-1)+\tilde{g}_m(x),\ee where
\be\label{e38} \tilde{g}_m(x):=\sum_{j=1}^{m} a_j\vp_j(x).\ee
Choose $\tilde{g}_m(x)$ so
that \be\label{e39} \max_{x\in[-1,1]}|g(x)-\tilde{g}_m(x)|\to
0\text{ as }m\to \infty. \ee

Then \be\label{e40} G_n(y_m)\leq
\|\mathcal{H}_{n}(f-R\tilde{h}_m)\|^2_{w^{(n)},1}, \ee because $y_m$
is the unique minimizer of $G_n$.\\
Let us prove that $\|\mathcal{H}_n(f-R\tilde{h}_m)\|^2_{w^{(n)},1}\to
0$ as $n\to \infty$. Let
$$W_{1,m}(x):=f(x)-R\tilde{h}_m(x), \qquad
W_{2,m}:=f'(x)-(R\tilde{h}_m)'(x).$$ Then \be\label{e41}
W_{1,m}(x)=e^{-x}\int_{-1}^xe^{y}(g(y)-\tilde{g}_m(y))dy+
e^{x}\int_{x}^1e^{-y}(g(y)-\tilde{g}_m(y))dy
\ee and
\be\begin{split}\label{e42}
W_{2,m}(x)&=e^x\int_{x}^1e^{-y}(g(y)-\tilde{g}_m(y))dy-
e^{-x}\int_{-1}^xe^{y}(g(y)-\tilde{g}_m(y))dy.
\end{split}\ee
Thus, the functions  $[W_{1,m}(x)]^2$ and $[W_{2,m}(x)]^2$ are
Riemann-integrable. Therefore, \be\label{e43}
\dl_n:=|\|f-R\tilde{h}_m\|_1^2-\|\mathcal{H}_{n}(f-R\tilde{h}_m)\|^2_{w^{(n)},1}|\to
0 \text{ as } n\to \infty. \ee  Formula \eqref{e43} and  the
triangle inequality yield \be\label{e44}
\|\mathcal{H}_{n}(f-R\tilde{h}_m)\|^2_{w^{(n)},1}\leq
\dl_n+\|f-R\tilde{h}_m\|_1^2.\ee Let us derive an estimate for
$\|f-R\tilde{h}_m\|_1^2$. From \eqref{e41} and \eqref{e42} we obtain
the estimates: \be\begin{split}\label{e44a} |W_{1,m}(x)|&\leq
\max_{y\in
[-1,1]}|g(y)-\tilde{g}_m(y)|\left(e^{-x}\int_{-1}^xe^{y}dy+
e^{x}\int_{x}^1e^{-y}dy\right)\\
&=\max_{y\in [-1,1]}|g(y)-\tilde{g}_m(y)|\left[e^{-x}(e^x-e^{-1})+
e^{x}(e^{-x}-e^{-1})\right]\\
&=\max_{y\in [-1,1]}|g(y)-\tilde{g}_m(y)|\left[(2-e^{-1-x}
-e^{-1+x})\right]\leq \dl_{m,1}
\end{split}\ee
and \be\begin{split}\label{e44b} |W_{2,m}(x)|&\leq \max_{y\in
[-1,1]}|g(y)-\tilde{g}_m(y)|\left(e^x\int_{x}^1e^{-y}dy+
e^{-x}\int_{-1}^xe^{y}dy\right)\leq \dl_{m,1}, \end{split}\ee where
\be\label{e46} \dl_{m,1}:=2\max_{y\in[-1,1]}|g(y)-\tilde{g}_m(y)|.
\ee Therefore, it follows from \eqref{e44a} and \eqref{e44b} that
\be\begin{split}\label{e45} \|f-R\tilde{h}_m\|_1^2&=
\int_{-1}^1|W_{1,m}(x)|^2dx+\int_{-1}^1|W_{2,m}(x)|^2dx\\
&\leq 4\dl_{m,1}^2,
\end{split}\ee
where $\dl_{m,1}$ is defined in \eqref{e46}.  Using relation
\eqref{e39}, we obtain $\lim_{m\to \infty} \dl_{m,1}=0$. Since
$m=m(n)$ and $\lim_{n\to\infty}m(n)=\infty$, it follows from
\eqref{e44} and \eqref{e45} that
$\|\mathcal{H}_n(f-R\tilde{h}_m)\|^2_{w,1}\to 0$ as $n\to \infty.$
This together with \eqref{e40} imply $G_n(y_m)\to 0$ as $n\to
\infty$.\\
 \lemref{lem1} is proved.
\end{proof}

\begin{thm}\label{thm1}
Let the vector $c^{(m)}_{min}:=\left(
                      \begin{array}{c}
                        c_{-1}^{(m)} \\
                        c_{0}^{(m)} \\
                        c_{1}^{(m)} \\
                        \vdots \\
                        c_{m}^{(m)} \\
                      \end{array}
                    \right)\in \R^{m+2}
$ solve linear algebraic system \eqref{e29} and
$$h_m(x)=c_{-1}^{(m)}\dl(x+1)+c_{0}^{(m)}\dl(x-1)+\sum_{j=1}^{m}
c_{j}^{(m)}\vp_j(x).$$ Then \be\label{e48} \|h-h_m\|_{H^{-1}}\to
0\text{ as }n\to \infty. \ee
\end{thm}
\begin{proof}
We have \be\begin{split}\label{e49}
\|h-h_m\|_{H^{-1}[-1,1]}^2&=\|R^{-1}(f-Rh_m)\|_{H^{-1}[-1,1]}^2\\
&\leq
\|R^{-1}\|_{H^1[-1,1]\to \dot{H}^{-1}[-1,1]}^2\|f-Rh_m\|_1^2\\
&\leq C\left(G_n(c^{(m)}_{min})+\left|\|f-Rh_m\|_1^2-G_n(c^{(m)}_{min})\right|\right)\\
&\leq C[G_n(c^{(m)}_{min})+\dl_n]\to 0\text{ as } n\to \infty,
\end{split}
\ee where $C>0$ is a constant and \lemref{lem1} was used.\\
\thmref{thm1} is proved.
\end{proof}

\section{The choice of collocation points and basis functions}
In this section we give an example of the collocation points $x_i$,
$i=1,2,\hdots,n$, and basis functions $\vp_j$, $j=1,2,\hdots,m$,
such that the vectors $\mathcal{H}_nq_j$, $j=-1,0,1,\hdots,m,$ are
linearly independent, where $m+2\leq n$, \be\label{vHq10}
\mathcal{H}_nq_{-1}=\left(
                      \begin{array}{c}
                        e^{-1-x_1} \\
                        e^{-1-x_2} \\
                        \vdots \\
                        e^{-1-x_n} \\
                      \end{array}
                    \right),\quad \mathcal{H}_nq_{0}=\left(
                      \begin{array}{c}
                        e^{-1+x_1} \\
                        e^{-1+x_2} \\
                        \vdots \\
                        e^{-1+x_n} \\
                      \end{array}
                    \right)
  \ee and
\be\label{vHqj} \mathcal{H}_nq_{j}=\left(
                      \begin{array}{c}
                        e^{-x_1}\int_{-1}^1e^y\vp_j(y)dy+e^{x_1}\int_{-1}^1e^{-y}\vp_j(y)dy \\
                        e^{-x_2}\int_{-1}^1e^y\vp_j(y)dy+e^{x_2}\int_{-1}^1e^{-y}\vp_j(y)dy \\
                        \vdots \\
                        e^{-x_n}\int_{-1}^1e^y\vp_j(y)dy+e^{x_n}\int_{-1}^1e^{-y}\vp_j(y)dy \\
                      \end{array}
                    \right),\ j=1,2,\hdots,m.
\ee

Let us choose $w_j^{(n)}$ and $x_j$ as in \eqref{e22} and
\eqref{e23}, respectively, with an even number $n\geq 6$.
\begin{figure}[htp]
\begin{center}
  \includegraphics[width=10cm]{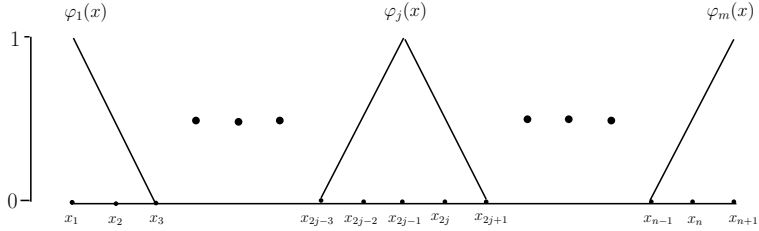}\\
  \caption{The structure of the basis functions $\vp_j$ }
\end{center}
\end{figure} As the basis functions in $C[-1,1]$ we choose the following linear
B-splines: \be\begin{split}\label{e50}
\vp_1(x)&=\left\{\begin{array}{ll} \psi_{2}(x) & \hbox{$x_{1}\leq
x\leq x_{3}$,}\\ 0, & \hbox{otherwise,}
                                                            \end{array}
                                                          \right.\\
\vp_j(x)&=\left\{
                                                            \begin{array}{ll}
                                                              \psi_{1}(x-(j-1)2s), & \hbox{$x_{2j-3}\leq x\leq x_{2j-1}$,} \\
                                                              \psi_{2}(x-(j-1)2s), & \hbox{$x_{2j-1}\leq x\leq x_{2j+1}$,}\\
                                                              0, & \hbox{otherwise,}
                                                            \end{array}
                                                          \right.\\
&j=2,\hdots,m-1,\\
\vp_{m}(x)&=\left\{\begin{array}{ll}
                                                              \psi_{1}(x-(m-1)2s), & \hbox{$x_{n-1}\leq x\leq 1$,} \\
                                                              0, & \hbox{otherwise,}
                                                            \end{array}
                                                         \right.\end{split}\ee

where $$m=\frac{n}{2}+1, \qquad s:=\frac{2}{n},$$ and
\be\begin{split}\label{e51}
\psi_{1}(x)&:=\frac{x-x_1+2s}{2s},\\
\psi_{2}(x)&:=\frac{-(x-x_1-2s)}{2s}.\end{split}\ee Here we have
chosen $x_{2j-1}$, $j=1,2,\hdots,\frac{n}{2}+1$, as the knots of the
linear B-splines. From Figure 1 we can see that at each
$j=2,\hdots,m-1,$ $\vp_j(x)$ is a "hat" function. The advantage of
using these basis functions is the following: at most two basis
functions are needed for computing the solution $h_m(x)$, because
\be\label{e52} h_m(x)=\left\{
                                                   \begin{array}{ll}
                                                     c_l^{(m)}, & \hbox{$x=x_{2l-1}$,} \\
                                                    c_l^{(m)}\vp_l(x)+c^{(m)}_{l+1}\vp_{l+1}(x), & \hbox{$x_{2l-1}<x<x_{2l+1}$,}
                                                   \end{array}
                                                 \right.\ l=2,\hdots,\frac{n}{2}.
 \ee

From the structure of the basis functions $\vp_j$ we have
\be\label{svp}\begin{split} \vp_1(x)&=0,\ x_3\leq x\leq 1,\\
\vp_j(x)&=0,\ -1\leq x\leq x_{2j-3} \text{ and } x_{2j+1}\leq x\leq 1,\ j=2,3,\hdots,m-1,\\
\vp_m(x)&=0,\ -1\leq x\leq x_{n-1}.
\end{split}\ee
Let $(\mathcal{H}_nq_j)_i$ be the $i$-th element of the vector
$\mathcal{H}_nq_j$, $j=-1,0,1,\hdots,m$. Then \be\label{Hqm1}
(\mathcal{H}_nq_{-1})_i=e^{-(1+x_i)},\ i=1,2,\hdots,n, \ee
\be\label{Hq0} (\mathcal{H}_nq_0)_i=e^{-(1-x_i)},\ i=1,2,\hdots,n.
\ee Using \eqref{svp} in \eqref{vHqj}, we obtain
\be\label{Hqj}\begin{split} (\mathcal{H}_nq_1)_{i}&=\left\{
                                                     \begin{array}{ll}
                                                       \frac{2s-1 + e^{-2s}}{2s}, & \hbox{$i=1$,} \\
                                                       1-e^{-s}, & \hbox{$i=2$,}\\
                                                       e^{1-(i-1)s}C_1, & \hbox{$3\leq i\leq n$,}
                                                     \end{array}
                                                   \right. \\
 (\mathcal{H}_nq_j)_{i}&=e^{-1+(i-1)s}D_j, \ 1\leq
i\leq 2j-3,\ j=2,3,\hdots,m-1, \\
(\mathcal{H}_nq_j)_{2j-2}&=\frac{e^{-3 s} - e^{-s} + 2 s}{2 s},\
j=2,3,\hdots,m-1, \\
(\mathcal{H}_nq_j)_{2j-1}&=\frac{-1 + e^{-2 s} + 2 s}{s},\
j=2,3,\hdots,m-1, \\
(\mathcal{H}_nq_j)_{2j}&=(\mathcal{H}_nq_j)_{2j-2},\
j=2,3,\hdots,m-1,\\
(\mathcal{H}_nq_i)_{j}&= e^{1-(i-1)s}C_j,\ 2j+1\leq i\leq n,\
j=2,3,\hdots,m-1, \\
(\mathcal{H}_nq_m)_i&=\left\{
                                                     \begin{array}{ll}
                                                       e^{-1+(i-1)s}D_m, & \hbox{$1\leq i\leq n-1$,} \\
                                                       1-e^{-s}, & \hbox{$i=n$,}
                                                     \end{array}
                                                   \right.\end{split}\ee where

\be\begin{split}\label{Cj}
C_1&:=\int_{-1}^{x_3}e^{y}\vp_1(y)dy=\frac{-1
+ e^{2 s} - 2 s}{2 es},\\
C_j&:=\int_{x_{2j-3}}^{x_{2j+1}}e^y\vp_j(y)dy=\frac{e^{-1 + 2 (-2 +
j) s} (-1 + e^{2 s})^2}{2 s},\quad j=2,3,\hdots, m-1,
\end{split}\ee
\be\begin{split}\label{Dj}
D_1&:=\int_{-1}^{x_3}e^{-y}\vp_1(y)dy=\frac{e (2s-1 + e^{-2 s})}{2 s},\\
D_j&:=\int_{x_{2j-3}}^{x_{2j+1}}e^{-y}\vp_j(y)dy=\frac{e^{1 - 2 j s}
(-1 + e^{2 s})^2}{2 s},\quad
j=2,3,\hdots, m-1,\\
D_m&:=\int_{x_{n-1}}^1e^{-y}\vp_m(y)dy=\frac{-1 + e^{2 s} - 2 s}{2
es}.\end{split}\ee

\begin{thm}\label{thm13}
Consider $q_j$ defined in \eqref{e14} with $\vp_j$ defined in
\eqref{e50}. Let the collocation points $x_j$, $j=1,2,\hdots,n$, be
defined in \eqref{e23} with an even number $n\geq 6$. Then the
vectors $\mathcal{H}_{n}q_j$, $j=-1,0,1,2,\hdots,m$, $m=\frac{1}{s}+1$, $s=\frac{2}{n}$, are linearly independent, where $\mathcal{H}_n$ is
defined in \eqref{e15}.
\end{thm}
\begin{proof} Let
\be\begin{split} V_0&:=\{\mathcal{H}_{n}q_{-1},\ \mathcal{H}_{n}q_0\},\\
V_j&:=V_{j-1}\cup\{\mathcal{H}_{n}q_j\},\quad
j=1,2,\hdots,m.\end{split}\ee
We prove that the elements of the sets $V_j$, $j=0,1,\hdots,m$, are
linearly independent.

The elements of the set $V_0$ are linearly independent. Indeed,
$\mathcal{H}_nq_{j}\neq 0$ $\forall j$, and assuming that there
exists a constant $\alpha$ such that \be\label{L0}
\mathcal{H}_nq_{-1}=\alpha\mathcal{H}_nq_0,\ee one gets a
contradiction: consider the first and the $n$-th equations of
\eqref{L0}, i.e., \be\label{L01}
(\mathcal{H}_nq_{-1})_1=\alpha(\mathcal{H}_nq_0)_1 \ee and
\be\label{L0n} (\mathcal{H}_nq_{-1})_n=\alpha(\mathcal{H}_nq_0)_n,
\ee respectively.
 It
follows from \eqref{Hqm1}, \eqref{Hq0} and \eqref{L01} that
\be\label{al20}\alpha=e^2.\ee From \eqref{L0n}, \eqref{Hqm1},
\eqref{Hq0} and \eqref{al20} it follows that \be e^{-2+s}=e^{2-s}.\ee  This is a contradiction, which proves that $\mathcal{H}_nq_{-1}$
and $\mathcal{H}_nq_0$ are linearly independent.

Let us prove that the element of the set $V_j$ are linearly
independent, $j=1,2,3,\hdots,m-2$. Assume that there exist constants
$\alpha_k$, $k=1,2,\hdots,j+1$, such that \be\label{Lj}
\mathcal{H}_nq_j=\sum_{k=-1}^{j-1}\alpha_{k+2}\mathcal{H}_nq_k. \ee
Using relations \eqref{Hqm1}-\eqref{Hqj} one can write the $(2j-1)$-th
equation of linear system \eqref{Lj} as follows:
\be\label{Lj1}\begin{split}
(\mathcal{H}_nq_j)_{2j-1}&=\sum_{k=-1}^{j-1}\alpha_{k+2}\mathcal({H}_nq_k)_{2j-1}\\
&=\alpha_1e^{-(2j-2)s}+\alpha_2e^{-2+(2j-2)s}+e^{1-(2j-2)s}\sum_{k=1}^{j-1}\alpha_{k+2}C_{k}.
\end{split}\ee Similarly, by relations \eqref{Hqm1}-\eqref{Hqj} the $(n-1)$-th and $n$-th equations of linear
system \eqref{Lj} can be written in the following expressions:
\be\label{Lj2}\begin{split}
(\mathcal{H}_nq_j)_{n-1}&=e^{-1+2s}C_j=\alpha_1e^{-2+2s}+\alpha_2e^{-2s}+e^{-1+2s}\sum_{k=1}^{j-1}\alpha_{k+2}C_k
\end{split}\ee and \be\label{Lj3}
(\mathcal{H}_nq_j)_{n}=e^{-1+s}C_j=\alpha_1e^{-2+s}+\alpha_2e^{-s}+e^{-1+s}\sum_{k=1}^{j-1}\alpha_{k+2}C_k,\ee
respectively. Multiply \eqref{Lj3} by $e^s$ and compare with
\eqref{Lj2} to conclude that $\alpha_2=0$. From \eqref{Lj3} with
$\alpha_2=0$ one obtains \be\begin{split}\label{al1j}
\alpha_1=eC_j-e\sum_{k=1}^{j-1}\alpha_{k+2}C_k.
\end{split}\ee
Substitute $\alpha_1$ from \eqref{al1j} and
$\alpha_2=0$ into \eqref{Lj1} and get \be\label{I2Cj}
(\mathcal{H}_nq_j)_{2j-1}=e^{1-(2j-2)s}C_j. \ee
From \eqref{Cj} and
\eqref{Hqj} one  obtains for $0<s<1 ,\ j=1,2,3,\hdots,m-2$, the following
relation
\be\begin{split}\label{CjI2}
e^{1-(2j-2)s}C_j-(\mathcal{H}_nq_j)_{2j-1}&=\frac{e^{-2s} (-1 + e^{2
s})^2}{2 s}-\frac{-1 + e^{-2 s} + 2
s}{s}\\
&=\frac{e^{2s}-e^{-2s}-4s }{2s}=\frac{\sinh(2s)-2s}{s}>0,\end{split}
\ee which contradicts relation \eqref{I2Cj}. This contradiction
proves that the elements of the set $V_j$ are linearly independent,
$j=1,2,3,\hdots,m-2$, for $0<s<1$.

Let us prove that the elements of the set $V_{m-1}$, are linearly
independent. Assume that there exist constants $\alpha_{k}$,
$k=1,2,\hdots,m$, such that \be\label{Ln3}
\mathcal{H}_nq_{m-1}=\sum_{k=-1}^{m-2}\alpha_{k+2}\mathcal{H}_nq_k.
\ee  Using \eqref{Hqm1}-\eqref{Hqj}, the $(n-3)$-th equation of
\eqref{Ln3} can be written as follows: \be\label{hn31}\begin{split}
(\mathcal{H}_nq_{m-1})_{n-3}=&e^{1-4s}D_{m-1}=\sum_{k=-1}^{m-3}\alpha_{k+2}(\mathcal{H}_nq_k)_{n-3}\\
=&\alpha_1e^{-2+4s}+\alpha_2e^{-4s}+e^{-1+4s}\sum_{k=1}^{m-3}\alpha_{k+2}C_k+\alpha_m(\mathcal{H}_nq_{m-2})_{n-3}.
\end{split}\ee
Similarly we obtain the
$(n-2)$-th, $(n-1)$-th and $n$-th equations, corresponding to vector
equation \eqref{Ln3}:
\be\label{hn32}\begin{split}
(\mathcal{H}_nq_{m-1})_{n-2}&=\alpha_1e^{-2+3s}+\alpha_2e^{-3s}+e^{-1+3s}\sum_{k=1}^{m-3}\alpha_{k+2}C_k+\alpha_m(\mathcal{H}_nq_{m-2})_{n-2},
\end{split}\ee
\be\label{hn33}\begin{split}
(\mathcal{H}_nq_{m-1})_{n-1}&=\alpha_1e^{-2+2s}+\alpha_2e^{-2s}+e^{-1+2s}\sum_{k=1}^{m-2}\alpha_{k+2}C_k
\end{split}\ee and \be\label{hn34}\begin{split}
(\mathcal{H}_nq_{m-1})_{n}&=\alpha_1e^{-2+s}+\alpha_2e^{-s}+e^{-1+s}\sum_{k=1}^{m-2}\alpha_{k+2}C_k,\end{split}\ee
respectively. Multiply \eqref{hn34} by $e^s$ and compare with
\eqref{hn33} to get \be\label{al13}\begin{split}
\alpha_2&=\frac{(\mathcal{H}_nq_{m-1})_{n-1}-e^{s}(\mathcal{H}_nq_{m-1})_{n}}{e^{-2s}-1}=\frac{1-e^{2s}+4se^{2s}-2se^{3s}}{2s(1-e^{2s})},
\end{split}\ee
where formula \eqref{Hqj} was used. Multiplying \eqref{hn34} by $e^{3s}$,
comparing with equation \eqref{hn31}, and using \eqref{al13}, we obtain
\be\label{alma}\begin{split}
\alpha_m&=\frac{e^{1-4s}D_{m-1}-e^{3s}(\mathcal{H}_nq_{m-1})_{n}-\alpha_2(e^{-4s}-e^{2s})}{(\mathcal{H}_nq_{m-2})_{n-3}-e^{-1+4s}C_{m-2}}\\
&=\frac{2 + 4 s - 2 e^s s + 4 e^{2 s} s - 2 e^{3 s} s +e^{4 s} (-2
+4s)}{-1 + e^{4 s} - 4 e^{2 s} s}.\end{split}\ee
Another expression for $\alpha_m$ is obtained by multiplying
\eqref{hn34} by $e^{2s}$ and comparing with \eqref{hn32}:
\be\begin{split}\label{almb}
\alpha_m&=\frac{(\mathcal{H}_nq_{m-1})_{n-2}-e^{2s}(\mathcal{H}_nq_{m-1})_{n}-\alpha_2(e^{-3s}-e^{s})}{(\mathcal{H}_nq_{m-2})_{n-2}-e^{-1+3s}C_{m-2}}\\
&=\frac{2 + 4 s - 4 e^s s + e^{2 s} (-2 + 4 s)}{-1 + e^{2 s} - 2 e^s
s},
\end{split}\ee where $\alpha_2$ is given in \eqref{al13}.

In deriving formulas  \eqref{alma} and \eqref{almb} we have used
the relation $m=\frac{1}{s}+1$ and equation
\eqref{Hqj}. Let us prove that equations \eqref{alma} and
\eqref{almb} lead to a contradiction. Define\be\begin{split} r_1&:=2
+ 4 s - 2 e^s s + 4 e^{2 s} s - 2 e^{3 s} s +e^{4 s}
(-2 + 4s),\\
 r_2&:=-1 + e^{4 s} - 4 e^{2 s} s,\\
r_3&:=2 + 4 s - 4 e^s s + e^{2 s} (-2 + 4 s),\\
r_4&:=-1 + e^{2 s} - 2 e^s s.\end{split}\ee Then from \eqref{alma}
and \eqref{almb} we get \be\label{r3r1} r_3r_2-r_1r_4=0.\ee We have
\be\label{r1r3}\begin{split} r_3r_2-r_1r_4&=2 e^s (-1 + e^s)^2 s (3
+ 4 s +(4 s-3)e^{2 s}- 2 se^s)>0\quad \text{for }
s\in(0,1).\end{split} \ee The sign of the right side of equality
\eqref{r1r3} is the same as the sign of
$3+4s+(4s-3)e^{2s}-2se^s:=\beta(s).$ Let us check that $\beta(s)>0$
for $s\in (0,1).$ One has $\beta(0)=0,$ $\beta'(0)=0,$
$\beta'(s)=4-2e^{2s}+8se^{2s}-2e^s-2se^s,$
$\beta''=4e^{2s}+16se^{2s}-4e^s-2se^s>0.$ If $\beta''(s)>0$ for
$s\in (0,1)$ and $\beta(0)=0,$ $\beta'(0)=0$, then $\beta(s)>0$ for
$s\in (0,1).$ Inequality \eqref{r1r3} contradicts relation
\eqref{r3r1} which proves that $\mathcal{H}_nq_j$,
$j=-1,0,1,2,\hdots, m-1$, are linearly independent.

Similarly, to prove that $\mathcal{H}_nq_j$, $j=-1,0,1,2,\hdots, m$,
are linearly independent, we assume  that there exist constants
$\alpha_{k}$, $k=1,2,\hdots,m+1$, such that \be\label{Lm}
\mathcal{H}_nq_{m}=\sum_{k=-1}^{m-1}\alpha_{k+2}\mathcal{H}_nq_k.
\ee
Using formulas \eqref{Hqm1}-\eqref{Hqj}, one  can write the $(n-5)$-th
equation: \be\label{hm21}\begin{split}
(\mathcal{H}_nq_m)_{n-5}&=e^{1-6s}D_{m}=\sum_{k=-1}^{m-1}\alpha_{k+2}(\mathcal{H}_nq_k)_{n-5}\\
&=\alpha_1e^{-2+6s}+\alpha_2e^{-6s}+e^{-1+6s}\sum_{j=1}^{m-4}\alpha_{j+2}C_j\\
&+\alpha_{m-1}(\mathcal{H}_nq_{m-3})_{n-5}+\alpha_me^{1-6s}D_{m-2}+\alpha_{m+1}e^{1-6s}D_{m-1},
\end{split}\ee

Similarly one obtains the  $(n-4)$-th,
$(n-3)$-th,  $(n-2)$-th,  $(n-1)$-th and $n$-th equations
corresponding to the vector equation  \eqref{Lm}:
\be\label{hm22}\begin{split}
(\mathcal{H}_nq_m)_{n-4}&=e^{1-5s}D_{m}=\alpha_1e^{-2+5s}+\alpha_2e^{-5s}+e^{-1+5s}\sum_{j=1}^{m-4}\alpha_{j+2}C_j\\
&+\alpha_{m-1}(\mathcal{H}_nq_{m-3})_{n-4}+\alpha_m(\mathcal{H}_nq_{m-2})_{n-4}+\alpha_{m+1}e^{1-5s}D_{m-1},
\end{split}\ee

\be\label{hm23}\begin{split}
(\mathcal{H}_nq_m)_{n-3}&=e^{1-4s}D_{m}=\alpha_1e^{-2+4s}+\alpha_2e^{-4s}+e^{-1+4s}\sum_{j=1}^{m-3}\alpha_{j+2}C_j\\
&+\alpha_m(\mathcal{H}_nq_{m-2})_{n-3}+\alpha_{m+1}e^{1-4s}D_{m-1},
\end{split}\ee
\be\label{hm24}\begin{split}
(\mathcal{H}_nq_m)_{n-2}&=e^{1-3s}D_{m}=\alpha_1e^{-2+3s}+\alpha_2e^{-3s}+e^{-1+3s}\sum_{j=1}^{m-3}\alpha_{j+2}C_j\\
&+\alpha_m(\mathcal{H}_nq_{m-2})_{n-2}+\alpha_{m+1}(\mathcal{H}_nq_{m-1})_{n-2},
\end{split}\ee \be\label{hm25}\begin{split}
(\mathcal{H}_nq_m)_{n-1}&=e^{1-2s}D_{m}=\alpha_1e^{-2+2s}+\alpha_2e^{-2s}+e^{-1+2s}\sum_{j=1}^{m-2}\alpha_{j+2}C_j\\
&+\alpha_{m+1}(\mathcal{H}_nq_{m-1})_{n-1}
\end{split}\ee and
\be\label{hm26}
(\mathcal{H}_nq_m)_{n}=\alpha_1e^{-2+s}+\alpha_2e^{-s}+e^{-1+s}\sum_{j=1}^{m-2}\alpha_{j+2}C_j+\alpha_{m+1}(\mathcal{H}_nq_{m-1})_{n},\ee
 respectively. Here we have used the
assumption $n\geq 6$. From \eqref{hm26} one gets \be\label{al2n2}
\alpha_1=(\mathcal{H}_nq_m)_{n}e^{2-s}-\alpha_2e^{2-2s}-e\sum_{k=1}^{m-2}\alpha_{k+2}C_k-\alpha_{m+1}(\mathcal{H}_nq_{m-1})_{n}e^{2-s}.\ee
If one substitutes \eqref{al2n2} into equations \eqref{hm25}, \eqref{hm24}
and \eqref{hm23}, then one obtains the following relations:
 \be\label{a21}
\alpha_2=p_1-p_2\alpha_{m+1},\ee \be\begin{split}\label{a22}
\alpha_2&=p_3-p_4\alpha_m-p_5\alpha_{m+1}
\end{split}\ee and
\be\begin{split}\label{a23}
\alpha_2&=p_6-p_7\alpha_m-p_8\alpha_{m+1} ,
\end{split}\ee respectively, where \be\begin{split}
p_1&:=\frac{e^{1-2s}D_{m}-(\mathcal{H}_nq_m)_{n}e^{s}}{e^{-2s}-1},\quad p_2:=\frac{(\mathcal{H}_nq_{m-1})_{n-1}-(\mathcal{H}_nq_{m-1})_{n}e^{s}}{e^{-2s}-1}, \\
p_3&:=\frac{e^{1-3s}D_m-(\mathcal{H}_nq_m)_{n}e^{2s}}{e^{-3s}-e^s},\quad p_4:=\frac{(\mathcal{H}_nq_{m-2})_{n-2}-e^{-1+3s}C_{m-2}}{e^{-3s}-e^s}, \\
p_5&:=\frac{(\mathcal{H}_nq_{m-1})_{n-2}-(\mathcal{H}_nq_{m-1})_{n}e^{2s}}{e^{-3s}-e^s},\quad p_6:=\frac{e^{1-4s}D_m-(\mathcal{H}_nq_m)_{n}e^{3s}}{e^{-4s}-e^{2s}}, \\
p_7&:=\frac{(\mathcal{H}_nq_{m-2})_{n-3}-e^{-1+4s}C_{m-2}}{e^{-4s}-e^{2s}},
\quad
p_8:=\frac{e^{1-4s}D_{m-1}-(\mathcal{H}_nq_{m-1})_{n}e^{3s}}{e^{-4s}-e^{2s}}.
\end{split}\ee

Another formula for  $\alpha_1$ one gets
from equation \eqref{hm23}:
\be\begin{split}\label{a11} \alpha_1&=
e^{3-8s}D_{m}-\alpha_2e^{2-8s}-e\sum_{j=1}^{m-3}\alpha_{j+2}C_j-\alpha_m(\mathcal{H}_nq_{m-2})_{n-3}e^{2-4s}\\
&-\alpha_{m+1}e^{3-8s}D_{m-1}.\end{split}\ee Substituting
\eqref{a11} into equations \eqref{hm22} and \eqref{hm21}, yields
\be\label{a21b}
\alpha_2=p_9-p_{10}\alpha_{m-1}-p_{11}\alpha_m-p_{12}\alpha_{m+1}\ee
and \be\label{a22b}
\alpha_2=p_9-p_{13}\alpha_{m-1}-p_{14}\alpha_m-p_{12}\alpha_{m+1}
,\ee respectively, where \be\begin{split} p_9&:=eD_m,\quad
p_{10}:=\frac{(\mathcal{H}_nq_{m-3})_{n-4}-e^{-1+5s}C_{m-3}}{e^{-5s}-e^{-3s}},\\
p_{11}&:=\frac{(\mathcal{H}_nq_{m-2})_{n-4}-e^s(\mathcal{H}_nq_{m-2})_{n-3}}{e^{-5s}-e^{-3s}},\quad
p_{12}:=eD_{m-1},\\
p_{13}&:=\frac{(\mathcal{H}_nq_{m-3})_{n-5}-e^{-1+6s}C_{m-3}}{e^{-6s}-e^{-2s}},
\quad
p_{14}:=\frac{e^{1-6s}D_{m-2}-e^{2s}(\mathcal{H}_nq_{m-2})_{n-3}}{e^{-6s}-e^{-2s}}.
\end{split}\ee

Let us prove that the equations \eqref{a21} and \eqref{a21b}
lead to a contradiction. From equations \eqref{a21} and \eqref{a22}
we obtain \be\label{alm}
\alpha_m=\frac{p_3-p_1}{p_4}+\frac{p_2-p_5}{p_4}\alpha_{m+1}. \ee
This together with \eqref{a23} and \eqref{a21} yield \be\label{alm1}
\alpha_{m+1}=\frac{\frac{p_3-p_1}{p_4}-\frac{p_6-p_1}{p_7}}{\frac{p_2-p_8}{p_7}-\frac{p_2-p_5}{p_4}}.
\ee

Equations \eqref{a21b}, \eqref{a22b} and \eqref{alm} yield \be
\alpha_{m-1}=\frac{p_{14}-p_{11}}{p_{10}-p_{13}}\alpha_m=\frac{p_{14}-p_{11}}{p_{10}-p_{13}}\left(
\frac{p_3-p_1}{p_4}+\frac{p_2-p_5}{p_4}\alpha_{m+1}\right).\ee This
together with \eqref{a21b} imply \be\label{a2l}
\alpha_2=p_9-\left(p_{10}\frac{p_{14}-p_{11}}{p_{10}-p_{13}}+p_{11}\right)\left(
\frac{p_3-p_1}{p_4}+\frac{p_2-p_5}{p_4}\alpha_{m+1}\right)-p_{12}\alpha_{m+1},
\ee where $\alpha_{m+1}$ is given in \eqref{alm1}. Let \be\label{La}
L_1:=p_9-\left(p_{10}\frac{p_{14}-p_{11}}{p_{10}-p_{13}}+p_{11}\right)\left(
\frac{p_3-p_1}{p_4}+\frac{p_2-p_5}{p_4}\alpha_{m+1}\right)-p_{12}\alpha_{m+1}
\ee and \be\label{Lb} L_2:=p_1-p_2\alpha_{m+1}. \ee Then, from
\eqref{a21} and \eqref{a21b} one gets \be\label{L1L} L_1-L_2=0.\ee
Applying formulas \eqref{Hqm1}-\eqref{Hqj} in \eqref{La} and
\eqref{Lb} and using the relation $e^s=\sum_{j=0}^\infty
\frac{s^j}{j!}$, we obtain \be\begin{split}\label{LL2}
L_1-L_2&=\frac{2 e^{4 s} [1 + e^{2 s}(-1 + s) + s] (-1 + e^{2 s} - 2
e^s s)}{(-1 +
    e^{2 s}) s [3 + 4 s - 2 e^s s + e^{2 s} (-3 + 4 s)]}\\
&=\frac{2 e^{4
s}\left[\sum_{j=3}^\infty\frac{2^{j}(\frac{1}{2}-\frac{1}{j})}{(j-1)!}s^j\right]\left[\sum_{j=2}^\infty\frac{2(\frac{2^{j}}{j+1}-1)}{j!}s^{j+1}\right]}
{(-1 +
    e^{2 s}) s  [3 + 4 s - 2 e^s s + e^{2 s} (-3 + 4 s)] }>0,
\end{split}\ee because $e^{2s}> 1$ for all $0<s<1$,
$2^{j}(\frac{1}{2}-\frac{1}{j})>0$ for all $j\geq 3$,
$2(\frac{2^{j}}{j+1}-1)>0$ for all $j\geq 2$ and $3 + 4 s - 2 e^s s
+ e^{2 s} (-3 + 4 s)>0$ which was proved in \eqref{r1r3}. Inequality
\eqref{LL2} contradicts relation \eqref{L1L} which proves that
$\mathcal{H}_nq_j$,
$j=-1,0,1,2,\hdots, m$, are linearly independent.\\
\thmref{thm13} is proved.
\end{proof}

\section{Numerical experiments}
Note that for all $w,u,v\in \R^n$ we have \be
\sum_{l=1}^nw_lu_lv_l=v^tWu, \ee where $t$ stands for transpose and
\be\label{e106} W:=\left(
                                                               \begin{array}{ccccc}
                                                                 w_1 &0 & \hdots &0 & 0 \\
                                                                 0 & w_2 & 0 & \hdots & 0 \\
                                                                 \vdots & \ddots &\ddots & \ddots &0 \\
                                                                 \vdots & \hdots & 0 & w_{n-1} & 0\\
                                                                 0 & 0 & \hdots &0 & w_n \\
                                                               \end{array}
                                                             \right).
 \ee
Then \be\begin{split}\label{e107}
DP&:=\|\mathcal{H}_n(f-Rh_m)\|_{w^{(n)},1}^2\\
&=\sum_{l=1}^nw_l^{(n)}[(f(x_l)-Rh_m(x_l))^2+(f'(x_l)-(Rh_m)'(x_l))^2]\\
&=[\mathcal{H}_n(f-Rh_m)]^tW\mathcal{H}_n(f-Rh_m)\\
&+[\mathcal{H}_n(f'-(Rh_m)')]^tW\mathcal{H}_n(f'-(Rh_m)'),
\end{split}\ee where $W$ is defined in \eqref{e106} with
$w_j=w_j^{(n)}$, $j=1,2,\hdots,n$, defined in \eqref{e22}. The
vectors $\mathcal{H}_nRh_m$ and $\mathcal{H}_n(Rh_m)'$ are computed
as follows.

Using \eqref{Hqm1}-\eqref{Hqj}, the vector $\mathcal{H}_nRh_m$ can
be represented by \be\label{e109} \mathcal{H}_nRh_m=S_mc^{(m)},\ee
where $c^{(m)}=\left(
                                                  \begin{array}{c}
                                                    c_{-1}^{(m)} \\
                                                    c_{0}^{(m)} \\
                                                    c_{1}^{(m)} \\
                                                    \vdots \\
                                                    c_{m}^{(m)} \\
                                                  \end{array}
                                                \right)$ and
$S_m$ is an $n\times (m+2)$ matrix with the following entries:
\be\begin{split} (S_m)_{i,1}&=(\mathcal{H}_nq_{-1})_{i},\quad
(S_m)_{i,2}=(\mathcal{H}_nq_0)_{i},\ i=1,2,\hdots,n,\\
(S_m)_{i,j}&=(\mathcal{H}_nq_{j-2})_{i}, \ i=1,2,\hdots,n,\
j=3,4,\hdots,m+2.
\end{split}\ee

The vector $\mathcal{H}_n(Rh_m)'$ is computed as follows. Let \be
J_{i,j}:=e^{x_i}\int_{x_i}^1e^{-y}\vp_j(y)dy-e^{-x_i}\int_{-1}^{x_i}e^{y}\vp_j(y)dy,\
i=1,2,\hdots,n-1,\ j=1,2,\hdots,m. \ee This together with
\eqref{e50} yield \be\begin{split}\label{e112} J_{i,1}&=\left\{
                                                     \begin{array}{ll}
                                                       e^{-1}\frac{e(2l-1 + e^{-2l})}{2l}, & \hbox{$i=1$,} \\
                                                       \frac{e^{-l} (1 - e^l + l)}{l}, & \hbox{$i=2$,}\\
                                                       -\frac{e^{-x_i}(-1 + e^{2l} - 2 l)}{2le }, & \hbox{$i\geq 3$,}
                                                     \end{array}
                                                   \right. \\
                                                   J_{i,j}&=\left\{
                      \begin{array}{ll}
                        e^{x_i}\frac{e^{1 - 2 j l} (-1 + e^{2 l})^2}{2 l}, & \hbox{$i\leq 2j-3$;} \\
                        \frac{2 + e^{-3 l} - 3e^{-l}}{2 l}, & \hbox{$i=2j-2$;} \\
                        0, & \hbox{$i=2j-1$;} \\
                        -\frac{2 + e^{-3 l} - 3 e^{-l}}{2 l}, & \hbox{$i=2j$;} \\
                        -e^{-x_i}\frac{e^{-1 + 2 (-2 + j) l} (-1 + e^{2 l})^2}{2 l}, & \hbox{$i\geq 2j+1$,}
                      \end{array}
                    \right.\ 1\leq i\leq n,\ 1<j<m, \\
J_{i,m}&=\left\{
                                                     \begin{array}{ll}
                                                       \frac{e^{x_i}(-1 + e^{2 l} - 2 l)}{2 le}, & \hbox{$i\leq n-1$,} \\
                                                       \frac{e^{-l} (-1 + e^l - l)}{l}, & \hbox{$i=n$,}
                                                     \end{array}
                                                   \right.\end{split} \ee where $l=\frac{2}{n}$.

Then, using \eqref{e112}, the vector $\mathcal{H}_n(Rh_m)'$ can be
rewritten as follows: \be \mathcal{H}_n(Rh_m)'=T_mc^{(m)}, \ee where
$c^{(m)}=\left(
                                                  \begin{array}{c}
                                                    c_{-1}^{(m)} \\
                                                    c_{0}^{(m)} \\
                                                    c_{1}^{(m)} \\
                                                    \vdots \\
                                                    c_{m}^{(m)} \\
                                                  \end{array}
                                                \right)$ and $T_m$ is an $n\times
(m+2)$ matrix with the following entries: \be\begin{split}
(T_m)_{i,1}&=-e^{-(1+x_i)},\quad
(T_m)_{i,2}=e^{-(1-x_i)},\ i=1,2,\hdots,n,\\
(T_m)_{i,j}&=J_{i,j-2}, \ i=1,2,\hdots,n,\ j=3,4,\hdots,m+2.
\end{split}\ee
We consider the following examples discussed in \cite{RAMM4}:
\begin{enumerate}
\item[(1)] $f(x)=-2+2\cos(\pi(x+1))$ with the exact solution $h(x)=-1+(1+\pi^2)\cos(\pi(x+1))$.
\item[(2)] $f(x)=-2e^{x-1}+\frac{2}{\pi}\sin(\pi(x+1))+2\cos(\pi(x+1))$ with the exact solution
 $h(x)=\frac{1}{\pi}\sin(\pi(x+1))+(1+\pi^2)\cos(\pi(x+1))$.
\item[(3)]$f(x)=\cos(\frac{\pi(x+1)}{2})+4\cos(2\pi(x+1))-1.5\cos(\frac{7\pi(x+1)}{2})$ with the exact solution
$h(x)=\frac{1}{2}(1+\frac{\pi^2}{4})\cos(\frac{\pi(x+1)}{2})+(2+8\pi^2)\cos(2\pi(x+1))-0.75(1+12.25\pi^2)\cos(\frac{7\pi(x+1)}{2})+1.75\dl(x+1)+2.25\dl(x-1)$.
\item[(4)] $f(x)=e^{-x}+2\sin(2\pi(x+1))$ with the exact solution $h(x)=(1+4\pi^2)\sin(2\pi(x+1))+(e-2\pi)\dl(x+1)+2\pi\dl(x-1)$.

\end{enumerate}

In all the above examples we have $f\in C^2[-1,1]$. Therefore, one
may use the basis functions given in \eqref{e50}. In each example we
compute the relative pointwise errors: \be\label{e115}
RPE(t_i):=\frac{|g_m(t_i)-g(t_i)| }{\max_{1\leq i\leq M}|g(t_i)|},
\ee where $g(x)$ and $g_m(x)$ are defined in \eqref{e6} and
\eqref{e6b}, respectively, and \be t_i:=-1+(i-1)\frac{2}{M-1},\quad
i=1,2,\hdots,M. \ee The algorithm can be written as follows.

\begin{itemize}
\item[Step 0.] Set $k=3$, $n=2k$, $m=\frac{n}{2}+1$, $\epsilon\in
(0,1)$ and $DP\geq 10$, where $DP$ is defined in \eqref{e107}.
\item[Step 1.] Construct the weights $w_j^{(n)}$, $j=1,2,\hdots,n$, defined in \eqref{e22}.
\item[Step 2.]Construct the matrix $A_{m+2}$ and the vector
$F_{m+2}$ which are defined in \eqref{e29}.
\item[Step 3.] Solve for $c:=\left(
                              \begin{array}{c}
                                c_{-1}^{(m)} \\
                                c_0^{(m)} \\
                                c_1^{(m)} \\
                                \vdots \\
                                c_m^{(m)}
                              \end{array}
                            \right)
$ the linear algebraic system $A_{m+2}c=F_{m+2}$.
\item[Step 4.] Compute \be
DP=\|\mathcal{H}_n(f-Rh_m)\|_{w^{(n)},1}^2.\ee
\item[Step 5.] If $DP>\epsilon$ then set $k=k+1$, $n=2k$
and $m=\frac{n}{2}+1$, and go to Step 1. Otherwise, stop the
iteration and use $h_m(x)=\sum_{j=-1}^{m}c_j^{(m)}\vp_j(x)$ as
the approximate solution, where $\vp_{-1}(x):=\dl(1+x)$,
$\vp_0(x):=\dl(x-1)$ and $\vp_j(x),\ j=1,2,\hdots,m$, are
defined in \eqref{e50} and $c_j^{(m)}$, $j=-1,0,1,\hdots,m$, are
obtained in Step 3.
\end{itemize}

In all the experiments the following parameters are used: $ M=200$
and $\epsilon=10^{-4}, \ 10^{-6},\ 10^{-8}$. We also compute the
relative error \be RE:=\max_{1\leq i\leq M}RPE(t_i), \ee where $RPE$
is defined in \eqref{e115}. Let us discuss the results of our
experiments.\\

\begin{table}[htp]\caption{Example $1$  }
\newcommand{\m}{\hphantom{$-$}}
\renewcommand{\tabcolsep}{.85pc} 
\renewcommand{\arraystretch}{1.2} \begin{center}
\begin{tabular}{ccccccc}
\hline
$n$ & $m$&$\epsilon$& $a_{-1}$&$c^{(m)}_{-1}$& $a_0$& $c^{(m)}_0$\\
\hline
$24 $& $13 $&$10^{-4}$& $0$&$8.234\times 10^{-5} $& $    0 $ & $  -5.826\times 10^{-3}    $ \\
$32 $& $17 $&$10^{-6}$& $0$&$3.172\times 10^{-5} $& $    0 $ & $  -1.202\times 10^{-3}    $ \\
$56 $& $29 $ & $10^{-8}$& $    0  $&$    3.974\times 10^{-6} $& $    0 $ & $  -6.020\times 10^{-5} $ \\

\hline \hline
$n$& $m$&$\epsilon$& $DP$ & $RE$& &  \\
\hline
$24 $& $13 $&$10^{-4}$&$    8.610\times 10^{-6}   $&$    3.239\times 10^{-2}  $ & & \\
$32 $& $17 $&$10^{-6}$&$    7.137\times 10^{-7}   $&$    1.554\times 10^{-2}  $ & & \\
$56 $& $29 $&$10^{-8}$& $    6.404\times 10^{-9}   $&$    4.337\times 10^{-3} $ & & \\
\hline
\end{tabular}\end{center}
\end{table}

\noindent \textbf{Example 1.} In this example the coefficients
$a_{-1}$ and $a_0$, given in \eqref{e5} and \eqref{e6},
respectively, are zeros. Our experiments show, see Table 1, that the
approximate coefficients $c^{(m)}_{-1}$ and $c^{(m)}_0$ converge to
$a_{-1}$ and $a_0$, respectively, as $\epsilon\to 0$. Here to get
$DP\leq 10^{-6}$, we need $n=32$ collocation points distributed
uniformly in the interval $[-1,1]$. Moreover, the matrix $A_{m+2}$
is of the size $19$ by $19$ which is  small. For $\epsilon=10^{-6}$
the relative error $RE$ is of order $10^{-2}$. The $RPE$ at the
points $t_j$ are distributed in the interval $[0,0.018)$ as shown in
Figure 2. In computing the approximate solution $h_m$ at the points
$t_i$, $i=1,2,\hdots,M$, one needs at most two out of $m=17$ basis
functions $\vp_j(x)$. The reconstruction of the continuous part of
the exact solution can be seen in Figure 1. One can see from this
Figure that for $\epsilon=10^{-6}$ the continuous part $g(x)$ of the
exact solution $h(x)$ can be recovered very well by the approximate
function $g_m(x)$ at the points $t_j$, $j=1,2,\hdots, M$.

\begin{figure}[htp]
\begin{center}
  \includegraphics[width=10cm]{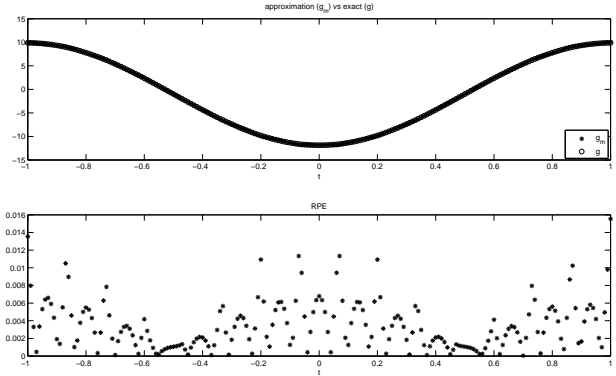}\\
  \caption{A reconstruction of the continuous part $g(x)$ (above) of Example 1 with $\epsilon=10^{-6}$, and the corresponding Relative Pointwise Errors (RPE) (below) }\label{fg1}
\end{center}
\end{figure}

\begin{table}[htp]\caption{Example $2$  }
\newcommand{\m}{\hphantom{$-$}}
\renewcommand{\tabcolsep}{.85pc} 
\renewcommand{\arraystretch}{1.2} \begin{center}
\begin{tabular}{ccccccc}
\hline
$n$ & $m$&$\epsilon$& $a_{-1}$&$c^{(m)}_{-1}$& $a_0$& $c^{(m)}_0$\\
\hline
$24 $& $13 $ & $  10^{-4}$& $    0   $&$    1.632\times 10^{-4} $& $    0 $ & $  -1.065\times 10^{-2} $ \\
$32 $& $17 $ & $  10^{-6}$& $    0   $&$    5.552\times 10^{-5} $& $    0 $ & $  -2.614\times 10^{-3} $ \\
$56 $& $29 $ & $10^{-8}$& $    0   $&$    6.268\times 10^{-6} $& $   0 $ & $  -1.954\times 10^{-4} $ \\

\hline \hline
$n$& $m$&$\epsilon$& $DP$ & $RE$& &  \\
\hline
$24 $& $13 $&$  10^{-4}$& $    8.964\times 10^{-6}   $&$    3.871\times 10^{-2} $ & & \\
$32 $& $17 $&$  10^{-6}$& $    7.588\times 10^{-7}   $&$    1.821\times 10^{-2} $ & & \\
$56 $& $29 $&$  10^{-8}$& $    6.947\times 10^{-9}   $&$    4.869\times 10^{-3} $& &  \\
\hline
\end{tabular}\end{center}
\end{table}

\noindent \textbf{Example 2.} This example is a modification of
Example 1, where the constant $-2$ is replaced with the function
$-2e^{x-1}+\frac{2}{\pi}\sin(\pi(x+1))$. In this case the
coefficients $a_{-1}$ and $a_0$ are also zeros. The results can be
seen in Table 2. As in Example 1, both approximate coefficients
$c^{(m)}_{-1}$ and $c^{(m)}_0$ converge to 0 as $\epsilon\to 0$. The
number of collocation points at each case is equal to the number of
collocation points obtained in Example 1. Also the $RPE$ at each
observed point is in the interval $[0,0.02)$. One can see from
Figure 3 that the continuous part $g(x)$ of the exact solution
$h(x)$ can be well approximated by the approximate function $g_m(x)$
with
$\epsilon=10^{-6}$ and $RE=O(10^{-2}).$\\

\begin{figure}[htp]
\begin{center}
  \includegraphics[width=10cm]{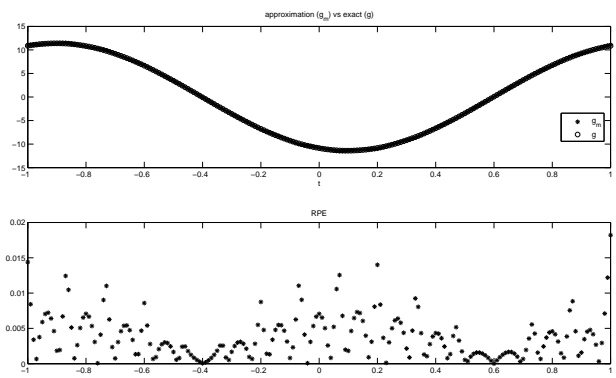}\\
  \caption{A reconstruction of the continuous part $g(x)$ (above) of Example 2 with $\epsilon=10^{-6}$, and the corresponding Relative Pointwise Errors (RPE) (below) }\label{fg2}
\end{center}
\end{figure}

\begin{table}[htp]\caption{Example $3$  }
\newcommand{\m}{\hphantom{$-$}}
\renewcommand{\tabcolsep}{.85pc} 
\renewcommand{\arraystretch}{1.2} \begin{center}
\begin{tabular}{ccccccc}
\hline
$n$ & $m$&$\epsilon$& $a_{-1}$&$c^{(m)}_{-1}$& $a_0$& $c^{(m)}_0$\\
\hline
$80 $& $41 $ & $  10^{-4}$& $    1.750   $&$    1.750 $& $    2.250 $ & $   2.236 $ \\
$128 $& $65 $ & $  10^{-6}$& $    1.750   $&$    1.750 $& $    2.250 $ & $   2.249 $ \\
$232 $& $117 $ & $ 10^{-8}$& $    1.750   $&$    1.750 $& $    2.250 $ & $   2.250 $ \\
\hline \hline
$n$& $m$&$\epsilon$& $DP$ & $RE$& &  \\
\hline $80 $& $41 $&$  10^{-4}$& $    4.635\times 10^{-5}   $&$ 2.282\times 10^{-2} $& &  \\
$128 $& $65 $&$  10^{-6}$& $    9.739\times 10^{-7}   $&$    7.671\times 10^{-3} $ & & \\
$232 $& $117 $&$ 10^{-8}$& $    7.804\times 10^{-9}   $&$    2.163\times 10^{-3} $ \\
\hline
\end{tabular}\end{center}
\end{table}

\noindent\textbf{Example 3.} In this example the coefficients of the
distributional parts $a_{-1}$ and $a_0$ are not zeros. The function
$f$ is  oscillating more than the functions $f$ given in Examples 1
and 2, and the number of collocation points is larger than in the
previous two examples, as shown in Table 3. In this table one can
see that the approximate coefficients $c^{(m)}_{-1}$ and $c^{(m)}_0$
converge to $a_{-1}$ and $a_0$, respectively. The continuous part of
the exact solution can be approximated by the approximate function
$g_m(x)$ very well with $\epsilon=10^{-6}$ and $RE=O(10^{-3}) $ as
shown in Figure 4. In the same Figure one can see that the $RPE$ at
each observed point is distributed in the interval $[0,8\times
10^{-3}).$

\begin{figure}[htp]
\begin{center}
  \includegraphics[width=10cm]{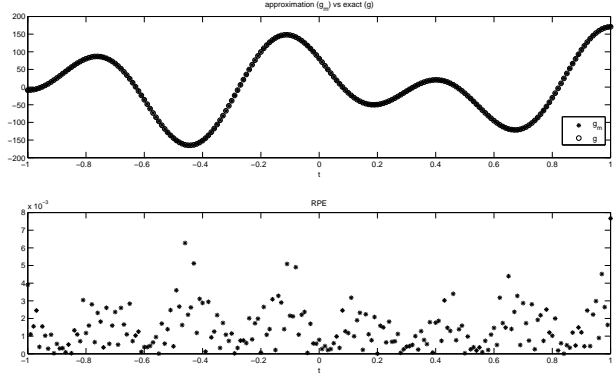}\\
  \caption{A reconstruction of the continuous part $g(x)$ (above) of Example 3 with $\epsilon=10^{-6}$, and the corresponding Relative Pointwise Errors (RPE) (below) }\label{fg3}
\end{center}
\end{figure}

\noindent \textbf{Example 4.} Here we give another example of the
exact solution $h$ having non-zero coefficients $a_{-1}$ and $a_0$.
In this example the function $f$ is oscillating less than the $f$ in
Example 3, but more than the $f$ in examples 1 and 2. As shown in
Table 4 the number of collocation points $n$ is smaller than the the
number of collocation points given in Example 3. In this example the
exact coefficients $a_{-1}$ and $a_0$ are obtained at the error
level $\epsilon=10^{-8}$ which is shown in Table 4. Figure 5 shows
that at the level $\epsilon=10^{-6}$ we have obtained a good
approximation of the continuous part $g(x)$ of the exact solution
$h(x)$. Here the relative error $RE$ is of order $O(10^{-2})$.

\begin{table}[htp]\caption{Example $4$  }
\newcommand{\m}{\hphantom{$-$}}
\renewcommand{\tabcolsep}{.85pc} 
\renewcommand{\arraystretch}{1.2} \begin{center}
\begin{tabular}{ccccccc}
\hline
$n$ & $m$&$\epsilon$& $a_{-1}$&$c^{(m)}_{-1}$& $a_0$& $c^{(m)}_0$\\
\hline
$40 $& $21 $ & $  10^{-4}$& $   -3.565   $&$   -3.564 $& $    6.283 $ & $   6.234 $ \\
$72 $& $37 $ & $  10^{-6}$& $   -3.565   $&$   -3.565 $& $    6.283 $ & $   6.279 $ \\
$128 $& $65 $ & $10^{-8}$& $   -3.565   $&$   -3.565 $& $    6.283 $ & $   6.283 $ \\
\hline \hline
$n$& $m$&$\epsilon$& $DP$ & $RE$& &  \\
\hline
$40 $& $21 $&$   10^{-4}$& $    8.775\times 10^{-5}   $&$    3.574\times 10^{-2} $ & & \\
$72 $& $37 $&$   10^{-6}$& $    6.651\times 10^{-7}   $&$    1.029\times 10^{-2} $ & & \\
$128 $& $65 $&$  10^{-8}$& $    6.147\times 10^{-9}   $&$    3.199\times 10^{-3} $ & &\\
\hline
\end{tabular}\end{center}
\end{table}

\begin{figure}[htp]
\begin{center}
  \includegraphics[width=10cm]{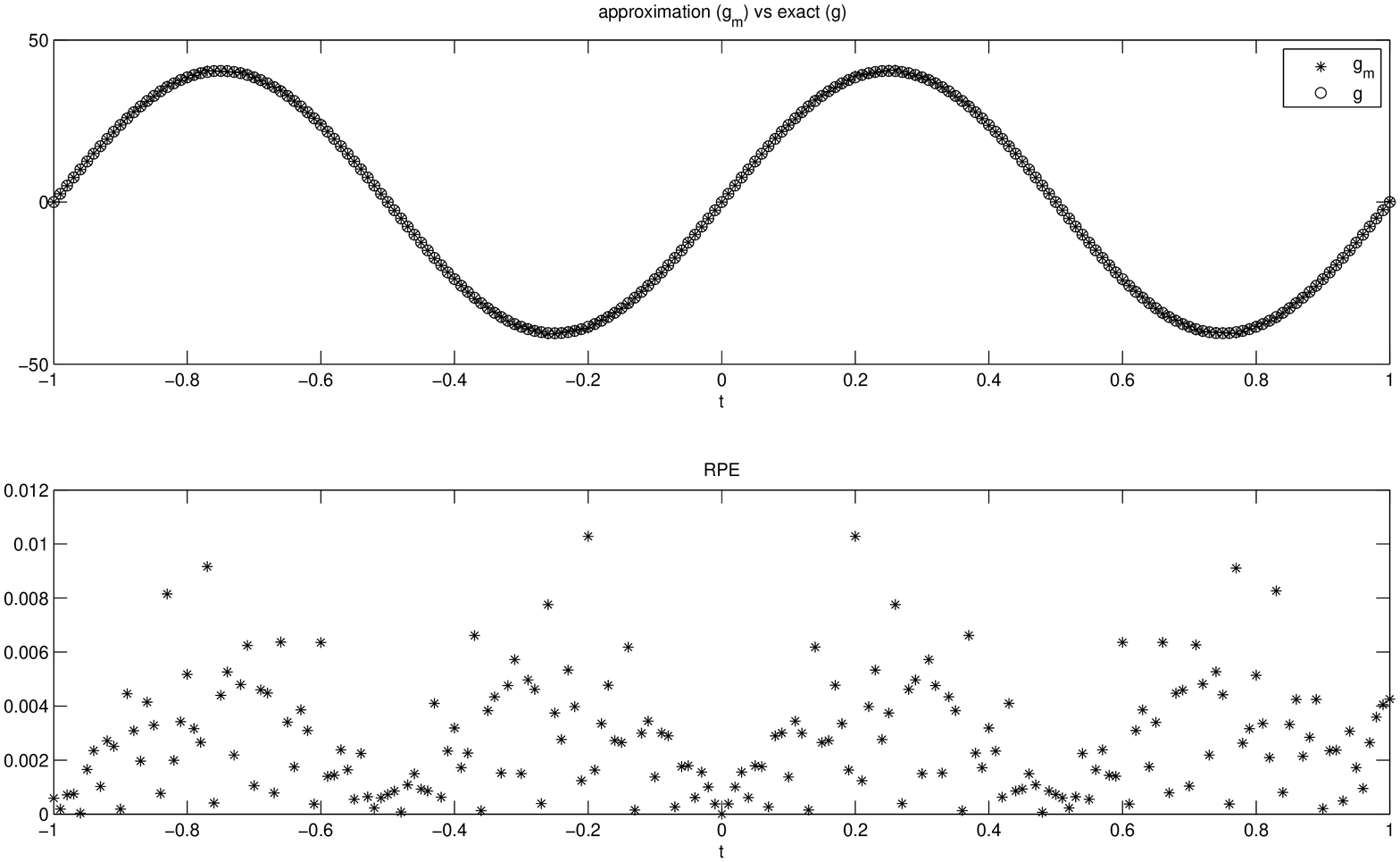}\\
  \caption{A reconstruction of the continuous part $g(x)$ (above) of Example 4 with $\epsilon=10^{-6}$, and the corresponding Relative Pointwise Errors (RPE) (below) }\label{fg4}
\end{center}
\end{figure}

\end{document}